\documentclass[11pt]{article}

\usepackage{amsmath,amssymb,amsthm,mathtools,setspace,graphicx,array,tikz,url}
\usepackage[text={6.5in,8.4in}, top=1.3in]{geometry}
\newtheorem{theorem}{Theorem}[section]

\newtheorem{lemma}[theorem]{Lemma}
\newtheorem{corollary}[theorem]{Corollary}
\newtheorem{conjecture}[theorem]{Conjecture}

\begin{document}
	
\title{
	Path-monochromatic bounded depth rooted trees in (random) tournaments 
}
	
\author{
	Raphael Yuster
	\thanks{Department of Mathematics, University of Haifa, Haifa 3498838, Israel. Email: raphael.yuster@gmail.com\,.}
}
	
\date{}
	
\maketitle
	
\setcounter{page}{1}
	
\begin{abstract}
	An edge-colored rooted directed tree (aka arborescence) is path-monochromatic if every path in it is monochromatic.
	Let $k,\ell$ be positive integers. For a tournament $T$, let $f_T(k)$ be the largest integer such that
	every $k$-edge coloring of $T$ has a path-monochromatic subtree with at least $f_T(k)$ vertices and let
	$f_T(k,\ell)$ be the restriction to subtrees of depth at most $\ell$.
	It was proved by Landau that $f_T(1,2)=n$ and proved by Sands et al. that $f_T(2)=n$ where $|V(T)|=n$.
	Here we consider $f_T(k)$ and $f_T(k,\ell)$ in more generality, determine their extremal values in most cases, and in fact in all cases assuming the Caccetta-H\"aggkvist Conjecture. We also study the typical value of $f_T(k)$ and $f_T(k,\ell)$, i.e., when $T$ is a random tournament. 

\end{abstract}

\section{Introduction}

In this paper we mainly consider oriented graphs, which are digraphs without loops, digons, or parallel edges.
In particular, we consider tournaments, which are oriented complete graphs.
An  {\em arborescence} is an oriented tree with a designated root,
such that for any vertex in the tree, there is a directed path in the tree from the root to that vertex.
The {\em depth} of an arborescence is the length of a longest directed path in it.
An edge-colored arborescence is {\em path-monochromatic} if every directed path in it is monochromatic.
A vertex $u$ of an edge-colored digraph {\em monochromatically dominates} a vertex $v$ if there is
a monochromatic directed path from $u$ to $v$. 
Hence, the root of a path-monochromatic arborescence monochromatically dominates all vertices of the arborescence.
Solving a longstanding conjecture of Erd\H{o}s, Sands, Sauer, and Woodrow mentioned in \cite{SSW-1982}, it was proved by
Bousquet, Lochet, and Thomass{\'e} \cite{BLT-2019} that there is a function $s(k)$, such that the vertices of every $k$-edge colored tournament can be covered using at most $s(k)$ path-monochromatic arborescences;
an exponential lower bound for $s(k)$ was obtained in \cite{BDH-2019}. Unavoidably, some of the covering
arborescences may have large depth. For example, it was shown in \cite{SSW-1982} that $s(2)=1$ but it is not difficult to construct examples of $2$-edge colorings of arbitrary large $n$-vertex tournaments in which every spanning path-monochromatic arborescence is a Hamiltonian path (i.e., has depth $n-1$).
On the other hand, it was already observed by Landau \cite{landau-1953} that every tournament has a spanning arborescence of depth $2$ (whose root is called a {\em king}).

Given the above, how large can a path-monochromatic arborescence of depth at most $\ell$ be in a $k$-edge colored graph? More formally, let $k,\ell$ be positive integers. For a tournament $T$, let $f_T(k,\ell)$ be the largest integer such that every $k$-edge coloring of $T$ has a path-monochromatic arborescence of depth at most $\ell$ with at least $f_T(k,\ell)$ vertices.
If the depth is unrestricted, we denote the parameter by $f_T(k)$.
Clearly, $f_T(k,1)$ is just one larger the maximum out-degree of $T$, hence always at least $\lceil (n+1)/2 \rceil$ where $n=|V(T)|$. By Landau's observation mentioned above, we have that $f_T(1,2)=n$ (using a king as the root). As the cases $k=1$ or $\ell=1$ are settled, we shall assume hereafter that $k \ge 2$
and $\ell \ge 2$. By the result in \cite{SSW-1982} mentioned earlier, we have $f_T(2)=n$
and there are cases for which $f_T(2,n-2) < n$. Hence, starting with $k=2$, it is of interest to
determine $f_T(k,\ell)$ for every $\ell \ge 2$. Naturally, we shall be interested in extremal or typical cases.

As for the extremal case, let $f(k,\ell,n)$ and $f(k,n)$ be the minimum of $f_T(k,\ell)$ (resp. $f_T(k)$) taken over all $n$-vertex tournaments,
let $f(k,\ell)= \lim_{n \rightarrow \infty} f(k,\ell,n)/n$ (we will prove that the limit exists in most cases and that it exists in all cases assuming the Caccetta-H\"aggkvist Conjecture).
Respectively define $f(k)=\lim_{n \rightarrow \infty} f(k,n)/n$ (so recall from the above that $f(2)=1$).

As for the typical case, let $f^*(k,\ell,n)$ and $f^*(k,n)$ be the random variable $f_T(k,\ell)$
(resp. $f_T(k)$) where $T$ is a random $n$-vertex tournament.
It is worth noting that properties of all edge colorings of random tournaments were studied by several researchers. We mention here a related variant of $f^*(2,n)$ where instead of looking for path-monochromatic arborescences,
one searches for the longest possible monochromatic path in every $2$-edge coloring of a random tournament.
This problem was essentially resolved by Buci\'c, Letzter, and Sudakov \cite{BLS-2019}; they proved that with high probability, there is always such a path of length $\Omega(n/\sqrt{\log n})$, resolving a conjecture in \cite{BKS-2012}. This result was recently extended in \cite{BHLS-2020} from paths to arbitrary oriented trees which are not necessarily an arborescences.

\vspace*{3mm}
Our first result concerns $f(k,\ell)$ and $f(k)$. Before stating it, we recall the longstanding
Caccetta-H\"aggkvist Conjecture \cite{CH-1978,sullivan-2006} (hereafter CH for short):

\begin{conjecture}[Caccetta-H\"aggkvist \cite{CH-1978}]
If an $n$-vertex digraph has minimum out-degree at least $r$, then it has
a directed cycle of length at most $\lceil n/r \rceil$.
\end{conjecture}

\begin{theorem}\label{t:1}
	$ $\\
	(i) $f(2,2)$ exists and furthermore, $f(2,2) \in [\sqrt{7}-2\,,\,\frac{2}{3}]=[0.654...,0.666...]$.\\
	(ii) Assuming CH, $f(2,\ell)$ exists for all $\ell \ge 2$ and equals $\frac{\ell}{\ell+1}$.\\
	(iii) $f(k)$ exists for all $k \ge 3$ and equals $\frac{1}{2}$.
\end{theorem}

We next consider $f^*(k,\ell,n)$ and $f^*(k,n)$. Recall that ``with high probability'' (whp for short) means a probability that tends to $1$ as $n$ tends to infinity.
We separate our results to two: the case $k=2$ and the case $k \ge 3$.
For the latter, our next result shows that for all $k \ge 3$, $f^*(k,2,n)$ is whp significantly larger than $f(k,n)$
(which is about $n/2$ by Theorem \ref{t:1} item (iii)).
In other words, while there are $k$-edge colorings of certain tournaments in which path-monochromatic arborescences (of any depth) have at most about $n/2$ vertices, whp every $k$-edge coloring of a random tournament has significantly more than $n/2$ vertices even if restricted to depth $2$.
On  the other hand, we will show that $f^*(k,n)$
approaches $n/2$ with high probability, exponentially fast in $k$. In other words, as the number of colors increases, random tournaments, even without depth restriction, achieve a constant fraction more than $1/2$
but that constant approaches zero rapidly.

\begin{theorem}\label{t:2}
	For every $k \ge 3$ (a) there exists a positive constant $c_k$ such that whp $f^*(k,2,n) \ge (c_k+\frac{1}{2})n$. On the other hand, (b) whp:
	\begin{equation}\label{e:t2}
	\frac{f^*(k,n)}{n} - o(1) \le \begin{cases}
		\frac{1}{2} + \frac{1}{4\cdot 3^{k/3-1}}   &  \text{for}\ k \equiv 0 \bmod 3\,;\\
		\frac{1}{2} + \frac{1}{2\cdot 3^{(k-1)/3}} &  \text{for}\ k \equiv 1 \bmod 3\,;\\
		\frac{1}{2} + \frac{1}{8\cdot 3^{(k-5)/3}} &  \text{for}\ k \equiv 2 \bmod 3\,.
	\end{cases}
\end{equation}
\end{theorem}

Turning to the two-color case, recall first that $f(2,2,n)$ is at most $\frac{2}{3}n+o(1)$ by Theorem \ref{t:1} item (i). This, however, is false for $f^*(2,2,n)$ which is significantly larger whp.
\begin{theorem}\label{t:3}
	(a) Whp, $f^*(2,2,n) +o(n) \ge \frac{49}{72}n$.
	On the other hand, (b) whp $f^*(2,2,n) - o(n) \le \frac{3}{4}n$.
\end{theorem}
Notice that Theorem \ref{t:3} (a) states that we can take $c_2 = 13/72$ in the statement of  Theorem \ref{t:2} (a), when applied to $k=2$. However, the proof of Theorem \ref{t:3} (a) is significantly different from the proof of Theorem \ref{t:2} (a)
(using the method of Theorem \ref{t:2} in the case $k=2$ will give a very small $c_2$ which will not be enough to show
the clear distinction from $f(2,2,n))$. We can summarize parts (a) of Theorems \ref{t:2} and \ref{t:3} as saying that for any $k \ge 2$, $f^*(k,2,n)$ is whp significantly larger than $f(k,2,n)$.

We end the introduction with a conjecture regarding two-edge colorings of random tournaments and depth larger than $2$. Motivated in part by the fact that random tournaments have $\Theta(n^2)$ directed paths of length at most $3$ connecting any two vertices, we conjecture that whp,
a random tournament has the property that in any two-edge coloring, there is a vertex
that monochromatically dominates all other vertices using paths of depth at most $3$.
\begin{conjecture}
	$f^*(2,3,n)=n$ whp.
\end{conjecture}
The rest of this paper consists of two sections. In the next section we prove Theorem \ref{t:1}
and in the last section we consider random tournaments and prove Theorems \ref{t:2} and \ref{t:3}.

\section{Depth $\ell$ arborescences in $k$-edge colored tournaments}\label{sec:f2l}

In this section we prove Theorem \ref{t:1}. We break the proof into several lemmas, for some of which we need the following definition.
For two tournaments $T_1$ and $T_2$, the {\em lexicographic product} $T_1 \circ T_2$ is the tournament on vertex set $V(T_1) \times V(T_2)$ where $((u_1,v_1),(u_2,v_2))$
is an edge if $v_1 \neq v_2$ and $(v_1,v_2) \in E(T_2)$  or if $v_1=v_2$ and $(u_1,u_2) \in E(T_1)$.
Edges of the former case are called {\em outer} and edges of the latter case are called {\em inner}.
Notice that the directed graph consisting of the outer edges is the $|V(T_1)|$-blowup of $T_2$.
For a tournament $T$ we denote the $r$'th power of $T$ by $T^r=T^{r-1} \circ T$.

We may extend the lexicographic product $T_1 \circ T_2$ to edge-colored tournaments
by assigning an outer edge $((u_1,v_1),(u_2,v_2))$ the color of $(v_1,v_2)$ and an inner edge
$((u_1,v),(u_2,v))$ the color of $(u_1,u_2)$. We call this coloring of $T_1 \circ T_2$ the {\em product coloring}. Given an edge colored tournament $T$, and applying the product coloring repeatedly, we obtain
the {\em power coloring} of $T^r$ with respect to the original coloring of $T$.
Observe that if $T$ has no monochromatic directed cycle, neither does $T^r$.

Let $R_q$ be the following red-blue edge colored tournament on vertex set $[q]$,
defined in \cite{SSW-1982}.
(note: we use $R_q$ to denote both the tournament and its particular coloring now stated).
It consists of a Hamiltonian directed red path $1,2,\ldots,q$ while all other edges are blue and oriented from a higher indexed vertex to a lower indexed vertex. Observe that vertex $1$ can reach vertex $q$ only by
the monochromatic Hamiltonian red path. For all $i > 1$, there is no monochromatic path (of any length) from
$i$ to $i-1$. Hence, $f_{R_q}(2,q-2) = q-1$. Also observe that $R_q$ has no monochromatic directed cycle.
\begin{lemma}\label{l:lex}
	If $n$ is a positive integer power of $\ell+2$, then $f(2,\ell,n) \le \frac{1+\ell n}{\ell+1}$.
\end{lemma}
\begin{proof}
	Let $n=(\ell+2)^r$ and consider $R_{\ell+2}^r$ together with its power coloring.
	We prove by induction on $r$ that $f_{R_{\ell+2}^r}(2,\ell)=\frac{1+\ell n}{\ell+1}$.
	For $r=1$ we use that $f_{R_{\ell+2}}(2,\ell) = \ell+1$, so the claim holds. Assume now that $r > 1$ and that we have already proved that
	\begin{equation}\label{e:1}
	f_{R_{\ell+2}^{r-1}}(2,\ell) = \frac{1+\ell(\ell+2)^{r-1}}{\ell+1}\;.
	\end{equation}
	Consider some vertex $(u,v)$ of $R_{\ell+2}^r$. Then $v \in V(R_{\ell+2}) = [\ell+2]$.
	If $v=1$ then $(u,1)$ cannot reach monochromatically with a path of length $\ell$ any vertex of the form $(x,\ell+2)$ because of the coloring of $R_{\ell+2}$ (to reach it, we must pass $\ell+1$ outer edges,
	all of which are red). If $v > 1$, then there is no monochromatic path (of any length)
	from $(u,v)$ to any vertex of the form $(x,v-1)$ because of the coloring of $R_{\ell+2}$.
	So, we see that in any case $(u,v)$ fails to reach at least $n/(\ell+2)=(\ell+2)^{r-1}$ vertices
	of the form $(u',v')$ for $v' \neq v$ via a monochromatic path of length at most $\ell$.
	Now consider monochromatic paths from $(u,v)$ to vertices of the form $(u',v)$. There are two options for such paths. Either each vertex on the path is also of the form $(u'',v)$, in which case such a path corresponds to a monochromatic path entirely within $R_{\ell+2}^{r-1}$. But then, by the induction hypothesis,
	such paths whose lengths are at most $\ell$ can reach at most the amount of vertices equal to the right hand side of \eqref{e:1}. Otherwise, as $R_q$ has no monochromatic directed cycles, the path uses at least two outer edges, at least one of which is red, and at least one of which is blue because of the coloring of $R_{\ell+2}$, but then the path is not monochromatic.
	Summarizing, we see that $(u,v)$ can reach, using monochromatic paths of length at most $\ell$,
	precisely
	$$
	\frac{1+\ell(\ell+2)^{r-1}}{\ell+1} + (\ell+2)^{r}-2(\ell+2)^{r-1} = \frac{1+\ell(\ell+2)^r}{\ell+1}=
	\frac{1+\ell n}{\ell+1}
	$$
	vertices. The lemma now follows since, by its definition, $f(2,\ell,n) \le 	f_{R_{\ell+2}^r}(2,\ell)$.
\end{proof}

\begin{lemma}\label{l:ch}
	Assuming CH, it holds that $f(2,\ell,n) \ge 1+\lfloor \frac{\ell}{\ell+1}n \rfloor$ 
\end{lemma}
\begin{proof}
	Suppose $T$ is a $2$-edge colored $n$-vertex tournament.
	Construct an oriented graph $F$ on $V(T)$ as follows. We have $(u,v) \in E(F)$ if and only if $u$ cannot reach $v$ with a monochromatic path of length at most $\ell$.
	If $F$ has a vertex $v$ with out-degree at most $\lceil n/(\ell+1)\rceil-1$ then $v$ monochromatically dominates (in $T$) at least  $1+\lfloor \frac{\ell}{\ell+1}n \rfloor$ vertices (including itself) via monochromatic paths of length at most $\ell$ and we are done.
	Otherwise, the minimum out-degree of $F$ is at least $\lceil n/(\ell+1) \rceil$.
	Then, assuming CH, $F$ has a directed cycle $C$ of length at most
	$\lceil n/\lceil n/(\ell+1) \rceil \rceil = \ell+1$.
	Let $X$ be the set of vertices of $C$. Then $T[X]$ is a $2$-edge colored sub-tournament of $T$ on at most $\ell+1$ vertices, and by the result of Sands et al. \cite{SSW-1982},
	$T[X]$ has a vertex $v$ that monochromatically dominates all other vertices of $T[X]$. In particular, there is a monochromatic path of length at most $\ell$ from $v$ to all other vertices of $T[X]$, contradicting the fact that $v$ cannot reach the vertex following it in $C$ with a path of length at most $\ell$.
\end{proof}
	
\begin{lemma}\label{l:3}
	If $n$ is a positive integer power of $3$, then $f(k,\ell,n)=f(3,\ell,n)=(n+1)/2$ for all $\ell \ge 1$.
\end{lemma}
\begin{proof}
	Let $C$ denote the directed triangle where each edge is colored with a distinct color. Now, let $n=3^r$ and consider $C^r$ together with its power coloring.
	Simple induction shows that $C^r$ is a regular tournament and that each vertex monochromatically dominates its out neighbors, and no other vertex. Hence, $f_{C^r}(3,\ell)=(n+1)/2$ for all $\ell \ge 1$.
	Finally, observe that we always have $(n+1)/2 \le f(k,\ell,n) \le f(3,\ell,n)$ since 
	every tournament has a vertex with out-degree at least $(n-1)/2$.
\end{proof}

\begin{lemma}\label{l:limit}
	For all $k \ge 3$ and all $\ell \ge 2$, $f(k,\ell)$ exists and equals $\frac{1}{2}$.
\end{lemma}
\begin{proof}
	Clearly, $\liminf_{n \rightarrow \infty} f(k,\ell,n)/n \ge \frac{1}{2}$ (recall, there is always a vertex with out-degree at least $(n-1)/2$).
	We must prove that for every $\varepsilon > 0$, it holds for all sufficiently large $n$ that
	$f(k,\ell,n) \le (\frac{1}{2}+\varepsilon)n$, as this will prove that
	$\limsup_{n \rightarrow \infty} f(k,\ell,n)/n \le \frac{1}{2}+\varepsilon$ and hence the claimed limit exists
	and equals $\frac{1}{2}$. Since $f(k,\ell,n) \le f(3,\ell,n)$, it suffices to prove this for $k=3$.
	
	Let $\varepsilon > 0$ and let $r$ be the smallest integer such that $3^r \ge n$. Consider the tournament $C^r$ from the proof of Lemma \ref{l:3}. It is a $3$-edge colored
	regular tournament with $3^r$ vertices, and by construction, each vertex monochromatically dominates its out-neighbors and no other vertex. Now, take a random subtournament $T$ of order $n$ of $C^r$. The expected out-degree of a vertex of $T$ is precisely $(n-1)/2$ and is hypergeometrically distributed
	$H(3^r-1,(3^r-1)/2,n-1)$. Hence, since $n \ge 3^r/3$ we have that for all
	sufficiently large $n$, with positive probability, $T$ has maximum out-degree at most $(1+\varepsilon)n/2$.
	As every vertex of $T$ monochromatically dominates only its out neighbors, we have that
	$f(3,\ell,n) \le (\varepsilon+\frac{1}{2})n$  for all sufficiently large $n$, as required.
\end{proof}
	
\begin{lemma}\label{l:limit-2}
	Assuming CH, for all $\ell \ge 2$ it holds that $f(2,\ell)$ exists and equals $\frac{\ell}{\ell+1}$.
\end{lemma}
\begin{proof}
	By Lemma \ref{l:ch}, $\liminf_{n \rightarrow \infty} f(2,\ell,n)/n \ge \frac{\ell}{\ell+1}$ (here is where we assume CH). 
	Let $\varepsilon > 0$ and let $r$ be the smallest integer such that $(\ell+2)^r \ge n$.
	Consider the tournament $R_{\ell+2}^r$ together with its power coloring.
	As shown in Lemma \ref{l:lex}, we have $f_{R_{\ell+2}^r}(2,\ell)=(1+\ell (\ell+2)^r)/(\ell+1)$.
	So every vertex cannot reach at least $((\ell+2)^r-1)/(\ell+1)$ vertices via monochromatic paths of length at most $\ell$.
	Now, take a random subtournament $T$ of order $n$ of $R_{\ell+2}^r$. For a vertex $v$ of $T$, the expected
	number of unreachable vertices via monochromatic paths of length at most $\ell$ starting at $v$ is
	precisely $(n-1)/(\ell+1)$ and is hypergeometrically distributed $H((\ell+2)^r-1),(\ell+2)^r-1)/(\ell+1),n-1)$. Since $n \ge (\ell+2)^r/(\ell+2)$, we have that for all
	sufficiently large $n$, with positive probability, every vertex of $T$ cannot reach at least
	$(1-\varepsilon)n/(\ell+1)$ vertices via monochromatic paths of length at most $\ell$.
	Thus, $f(2,\ell,n) \le (1+\varepsilon)n\ell/(\ell+1)$ for all sufficiently large $n$. Thus, 
	$\limsup_{n \rightarrow \infty} f(2,\ell,n)/n \le \frac{\ell}{\ell+1}$ and the claimed limit exists.
\end{proof}
	The existence of $f(2,2)$ may be proved independently of CH.
\begin{lemma}\label{l:limit-3}
	$f(2,2)$ exists.
\end{lemma}
\begin{proof}
	Let $s = \liminf_{n \rightarrow \infty} f(2,2,n)/n$
	and let $\varepsilon > 0$. By the definition of $\liminf$, there exists
	$q \ge 1+ 1/\varepsilon$ and a tournament $Q$ with $q$ vertices such that
	$t+1 := f_Q(2,2) \le (\varepsilon+s)q$. In particular, there is a red-blue edge coloring of $Q$, denoted by $C(Q)$, such that every path-monochromatic arborescence of depth at most $2$ has at most $t+1$ vertices.
	
	Consider the tournament $Q^r$ and its power coloring with respect to $C(Q)$.
	We claim that this coloring shows that $f_{Q^r}(2,2) \le 1+t(q^r-1)/(q-1)$.
	Proceeding by induction, for $r=1$ this just follows by the fact that $t+1=f_Q(2,2)$.
	As for the induction step, let $(u,v)$ be some vertex of $Q^r$ (recall that $u \in Q^{r-1}$ and
	$v \in Q$). Consider first monochromatic paths of length at most $2$ from $(u,v)$ to
	$(u',v')$ where $v \neq v'$. By the definition of $C(Q)$, there are at most $t q^{r-1}$ vertices reachable by such paths.
	Consider next monochromatic paths of length at most $2$ from $(u,v)$ to some vertex 
	$(u',v)$. Clearly, every inner vertex on such a path must also be of the form $(u'',v)$ as there are no paths of length $2$ at all from $(u,v)$ to $(u',v)$ using vertices of the form $(u'',v')$ with $v' \neq v$ as an inner vertex (as $Q$ is a tournament, it does not have cycles of length $2$). So, by the induction hypothesis applied to
	$Q^{r-1}$, there are at most $1+t(q^{r-1}-1)/(q-1)$ vertices reachable by such paths. Hence,
	$$
	f_{Q^r}(2,2) \le t q^{r-1} + 1+t(q^{r-1}-1)/(q-1) = 1+t(q^r-1)/(q-1)\;.
	$$
	We may now proceed as in Lemma \ref{l:limit-2}: let $r$ be the smallest integer such that $q^r \ge n$.
	Consider the tournament $Q^r$ together with its power coloring with respect to $C(Q)$.
	As $f_{Q^r}(2,2) \le 1+t(q^r-1)/(q-1)$, every vertex cannot reach at least $q^r-1-t(q^r-1)/(q-1)$ vertices via monochromatic paths of length at most $2$.
	Now, take a random subtournament $T$ of order $n$ of $Q^r$. For a vertex $v$ of $T$, the expected
	number of unreachable vertices via monochromatic paths of length at most $2$ starting at $v$ is
	at least $(q^r-1-t(q^r-1)/(q-1))(n-1)/(q^r-1)$ and is hypergeometrically distributed
	$H(q^r-1,z,n-1)$ where $z \ge q^r-1-t(q^r-1)/(q-1)$. Since $n \ge q^r/q$ we have that for all
	sufficiently large $n$, with positive probability, every vertex of $T$ cannot reach at least
	$(1-\varepsilon)(q^r-1-t(q^r-1)/(q-1))(n-1)/(q^r-1)$ vertices via monochromatic paths of length at most $2$.
	Thus, for all sufficiently large $n$, $f(2,2,n) \le n-(1-\varepsilon)(q^r-1-t(q^r-1)/(q-1))(n-1)/(q^r-1)$.
	Finally, recalling that $t+1 \le  (\varepsilon+s)q$ and that $q-1 \ge \frac{1}{\varepsilon}$ we have that
	\begin{align*}
	 f(2,2,n) & \le n-(1-\varepsilon)(q^r-1-t(q^r-1)/(q-1))(n-1)/(q^r-1) \\
	 & = 1+\varepsilon(n-1)+(1-\varepsilon)\frac{t(n-1)}{q-1} \\
	 & \le \varepsilon n + 1 + \frac{tn}{q-1} \\
	 & \le \varepsilon n + 1 +\frac{(\varepsilon+s)qn}{q-1} \\
	 & \le 2\varepsilon n + (1+\varepsilon)(\varepsilon+s)n
	 \end{align*}
     implying that $\limsup_{n \rightarrow \infty} f(2,2,n)/n \le s$.	
\end{proof}
There are quite a few results that come rather close to the bound asserted in CH. Of interest to us is the result obtained by Shen \cite{shen-1998} who proved that if
a directed graph has minimum out-degree at least $(3-\sqrt{7})n$ then it contain a directed triangle.
An immediate corollary of the proof of Lemma \ref{l:ch} where we use $\ell=2$ and use Shen's bound instead of
the bound $\lceil n/3 \rceil$ of CH is that $f(2,2,n) \ge (\sqrt{7}-2)n$ (and this bound is independent of CH).
Now, together with Lemma \ref{l:lex} and Lemma \ref{l:limit-3} we get
\begin{corollary}\label{coro:f22}
	$f(2,2) \in [\sqrt{7}-2,\frac{2}{3}]=[0.645...,0.666...]$.
\end{corollary}
\begin{proof}[Proof of Theorem \ref{t:1}]
	Item (i) follows from Lemma \ref{l:limit-3} and Corollary \ref{coro:f22}.
	Item (ii) follows from Lemma \ref{l:limit-2}. Item (iii) follows from Lemma \ref{l:limit}.
\end{proof}

\section{Random tournaments}\label{sec:random}

Recall that a random $n$-vertex tournament is the probability space ${\cal T}_n$ of tournaments on vertex set $[n]$ such that the direction of each edge is independently chosen with a fair coin flip.
In this section we prove Theorems \ref{t:2} and \ref{t:3} both of which claim bounds for the values of the random variables $f^*(k,\ell,n)$ and $f^*(k,n)$ that hold whp. Recall that $f^*(k,\ell,n) = f_T(k,\ell)$ and $f^*(k,n)=f_T(k)$ where $T \sim {\cal T}_n$.

A {\em balanced $t$-partition} of $T$ is a partition of its vertex set into $t$ equal parts
$V_1,\ldots,V_t$ (we assume hereafter that $t$ divides $n$ as this assumption has no bearing on the asymptotic claims). For $v \in V(T)$, let $g(v)$ denote the number of out-neighbors of $v$ that belong to the parts not containing $v$. For $T \sim {\cal T}_n$ we have that $g(v)$ is distributed ${\rm Bin}(\frac{t-1}{t}n,\frac{1}{2})$. Thus, with probability at least, say, $1-1/n^2$, $g(v) = (1+o(1))\frac{t-1}{2t}n$.
Also, observe that $T[V_i]$ is a random tournament on $n/t$ vertices, i.e., $T[V_i] \sim {\cal T}_{n/t}$.
The following lemma establishes Theorem \ref{t:2} (b).
\begin{lemma}\label{l:func}
	Let $r_1:=1$ and for $k \ge 2$ let $r_k :=\min_{t=1}^{k-1} \{\frac{t-1}{2t}+\frac{r_{k-t}}{t}\}$.
	Then with probability $1-o(1)$, $f^*(k,n)/n-o(1) \le r_k$. In particular, Theorem \ref{t:2} (b) holds.
\end{lemma}
\begin{proof}
	We prove the first part of the lemma by induction on $k$ where the cases $k=1$ and $k=2$ are trivial (notice that $r_2=1$).
	Assume $k \ge 3$ and let $T \sim {\mathcal T}_n$. For $1 \le t \le k-1$, consider a balanced $t$-partition of $T$ into parts $V_1,\ldots,V_t$. Using the color palette $[k]$, we color the edges
	of $T$ as follows. The $g(v)$ edges connecting $v \in V_j$ to its out-neighbors in other parts are colored
	$j$. This leaves the $k-t$ colors $\{k-t+1,\ldots,k\}$ yet unused. For each $1 \le j \le t$, the sub-tournament $T[V_j]$ is edge-colored with these $k-t$ colors so as to minimize the largest size of a
	path-monochromatic arborescence. Now consider some vertex $v \in V_j$. By our coloring there is no monochromatic path from $v$ to any $u$ such that $u$ is in another part and $u$ is not an out-neighbor of $v$.
	Hence, $v$ monochromatically dominates only $g(v)$ vertices which are not in its part.
	Furthermore, by the coloring of $T[V_j]$, $v$ monochromatically dominates only $f_{T[V_j]}(k-t)$
vertices in its own part (notice that $v$ cannot use vertices outside its part to monochromatically reach
vertices in its part).  But recall that $g(v)=(1+o(1))\frac{t-1}{2t}n$ with probability at least $1-1/n^2$
so for all $v \in V(T)$ it holds with probability at least $1-1/n$ that $g(v)=(1+o(1))\frac{t-1}{2t}n$.
Furthermore, by the induction hypothesis, since $T[V_j] \sim {\mathcal T}_{n/t}$ it holds that
$f_{T[V_j]}(k-t)/(n/t) - o(1)  \le r_{k-t}$ with probability $1-o(1)$. Hence we have that with probability at least $1-1/n-o(t)=1-o(1)$ it holds that $f_T(k) \le (1+o(1))\frac{t-1}{2t}n + \frac{r_{k-t}n}{t}+o(n)$.
choosing $t$ which minimizes $\frac{t-1}{2t}+\frac{r_{k-t}}{t}$, the induction claim holds.

For the second part of the lemma, it remains to prove that for each $k \ge 3$, $r_k$ is at most
the value given in \eqref{e:t2}.
For $k=3$ we have $r_3 = \min\{1,\frac{3}{4}\}=\frac{3}{4}$.
For $k=4$ we have $r_4 = \min\{\frac{3}{4}, \frac{3}{4}, \frac{2}{3} \} = \frac{2}{3}$.
For $k=5$ we have $r_5 = \min\{\frac{2}{3}, \frac{5}{8}, \frac{2}{3}, \frac{5}{8} \} = \frac{5}{8}$.
All of these bounds match the ones given in \eqref{e:t2}.
Proceeding inductively, we assume that the claim holds for all values smaller than $k$ and prove for $k \ge 6$.
Assume first that $k \equiv 0 \bmod 3$. Using $t=2$ we have $r_k \le \frac{1}{4}+ r_{k-2}/2$
so by the induction hypothesis,
$$
r_k \le \frac{1}{4}+\frac{1}{2}\left(\frac{1}{2}+\frac{1}{2\cdot 3^{(k-3)/3}}\right) = \frac{1}{2}+\frac{1}{4 \cdot 3^{k/3-1}}\;.
$$
Assume next that $k \equiv 1 \bmod 3$. Using $t=3$ we have $r_k \le \frac{1}{3}+ r_{k-3}/3$
so by the induction hypothesis,
$$
r_k \le \frac{1}{3}+\frac{1}{3}\left(\frac{1}{2}+\frac{1}{2\cdot 3^{(k-4)/3}}\right) = \frac{1}{2}+\frac{1}{2 \cdot 3^{(k-1)/3}}\;.
$$
Finally, assume that $k \equiv 2 \bmod 3$. Using $t=3$ we have $r_k \le \frac{1}{3}+ r_{k-3}/3$
so by the induction hypothesis,
$$
r_k \le \frac{1}{3}+\frac{1}{3}\left(\frac{1}{2}+\frac{1}{8\cdot 3^{(k-8)/3}}\right) = \frac{1}{2}+\frac{1}{8 \cdot 3^{(k-5)/3}}\;.
$$	
\end{proof}
A similar idea as the one in Lemma \ref{l:func} may be used to prove Theorem \ref{t:3} (b).
\begin{proof}[Proof of Theorem \ref{t:3} (b).]
	We prove that $f^*(2,2,n)/n-o(1) \le \frac{3}{4}$ whp.
	Let $T \sim {\mathcal T}_n$. Consider a balanced $2$-partition of $T$ into parts $V_1,V_2$.
	Color the edges with endpoints in distinct parts red and color the edges with both endpoints in the same part blue. Notice that if $u$ and $v$ are in distinct parts and $v$ is not an out-neighbor of $u$, then $u$
	cannot reach $v$ via a monochromatic path of length $2$ (it may reach it via a monochromatic path of longer length). Hence, since whp each vertex $u$ has $(1-o(1))n/4$ non out-neighbors in the opposite part,
	we have that $f_T(2,2) \le 3n/4+o(n)$ whp.
\end{proof}

We now turn to proving Theorem \ref{t:2} (a). Let $\alpha=\alpha(k)$ be a positive constant to be set later. We need the following lemma whose proof immediately follows by applying Chernoff's inequality, hence omitted.
\begin{lemma}\label{l:bipartite}
	Let $T \sim {\mathcal T}_n$. Whp, the following holds for all ordered pairs of disjoint sets of vertices $A,B$ of $T$ of order $\alpha n$ each: The number of edges from $A$ to $B$ is at least
	$\alpha^2 n^2/2 - o(n^2)$. \qed
\end{lemma}
We hereafter assume that $T \sim {\mathcal T}_n$ satisfies the property in Lemma \ref{l:bipartite}
and also satisfies the property that the out-degree and in-degree of each vertex is $n/2-o(n)$ (observe that the latter property also holds whp as the out-degree and in-degree are each distributed ${\rm Bin}(n-1,\frac{1}{2})$).

\begin{proof}[Proof of Theorem \ref{t:2} (a).]
	Consider first the case where there is some vertex $v \in V(T)$ for which
	there is some color $c$ such that at least $\alpha n$ edges entering $v$ are colored $c$ and at least
	$\alpha n$ edges emanating from $v$ are colored $c$. Let $A$ denote a set of $\alpha n$ out-neighbors of $v$ such that for each $u \in A$, $(v,u)$ is colored $c$ and let $B$ denote a set of $\alpha n$ in-neighbors of $v$ such that for each $w \in B$, $(w,v)$ is colored $c$.
	By Lemma \ref{l:bipartite}, there are at least $\alpha^2 n^2/2 - o(n^2)$ edges going from $A$ to $B$.
	This means that there are at least $\alpha^2 n^2/2 - o(n^2)$ directed triangles of the form $(w,v,u)$
	where $(w,v)$ is colored $c$, $(v,u)$ is colored $c$ and $(u,w)$ is an edge (we assume nothing about its color). But then by pigeonhole, there is some specific $w \in B$ and at least $\alpha n/2-o(n)$ directed triangles of  the form $(w,v,u)$ where $u \in A$ and $(u,w)$ is an edge. Each such $u$ is not an out-neighbor of $w$ (it is an in-neighbor of $w$) but still reachable from $w$ by a monochromatic path of length $2$ 
	(namely the path $w,v,u$ both of whose edges are colored $c$). Thus, $w$ monochromatically dominates at
	least $n/2+\alpha n/2 -o(n)$ vertices, so $f_T(k,2) \ge n/2+\alpha n/2 -o(n)$.
	
	Consider next the case that for each $v \in V(T)$ and for each color $c$, if at least $\alpha n$ edges entering $v$ are colored $c$, then fewer than $\alpha n$ edges emanating from $v$ are colored $c$.
	Let the {\em signature} of a vertex $v$ be the subset $S \subseteq [k]$ such that for each $c\in S$, at least
	$\alpha n$ edges entering $v$ are colored $c$. Assuming $\alpha < 1/(2k)$, we have that $S \neq \emptyset$.
	As there are less than $2^k$ possible signatures, let $U \subseteq V(T)$ be a set of $n/2^k$
	vertices, all having the same signature $S$. Consider the subtournament $T[U]$, which has $(1-o(1))n^2/2^{2k+1}$ edges. As there are only $k$ colors, there is a color $c$ appearing in at least
	$(1-o(1))n^2/k2^{2k+1}$ of these edges. Now suppose that $\alpha < 1/k2^{k+1}$. Then, there is some vertex
	of $U$ which has at least $\alpha n$ edges colored $c$ entering it, so $c \in S$.
	Likewise, there is some vertex $u \in U$ having at least $\alpha n$ edges colored $c$ emanating from it.
	But since all vertices of $U$ share the same signature, there are also at least $\alpha n$ edges colored $c$ entering $u$, contradicting our assumption.
	
	Setting $c_k = \alpha/3$ we have shown that $f_T(k,2) \ge n/2+c_k n$ whp, so $f^*(k,2,n) \ge (c_k+\frac{1}{2})n$ whp, as required.
\end{proof}
Observe that by the proof of the last lemma we may take $c_k$ larger than $1/k2^{k+3}$.
On the other hand, Theorem \ref{t:2} (b) (i.e., Lemma \ref{l:func}) shows that we cannot take $c_k$
to be sub-exponential in $k$ (even without depth restriction). Moreover, it shows that the base in the exponent of $k$ in the definition of $c_k$ cannot be improved to smaller than $3^{1/3}$.

Finally, we turn to proving Theorem \ref{t:3} (a). To this end, we first need the following
result of Frankl and R\"odl on hypergraph matchings  \cite{FR-1985}. Recall that a $k$-uniform hypergraph is a collection of $k$-sets (the edges) of some $m$-set (the vertices). The {\em degree} $d(x)$ of a vertex $x$ in a hypergraph is the number of edges containing $x$ and the {\em co-degree} of a pair of distinct vertices $x,y$ is the number of edges containing both. A {\em matching} in a hypergraph is a set of pairwise disjoint edges.
\begin{lemma}[Frank and R\"odl \cite{FR-1985}]\label{l:fr}
	For an integer $k \geq 3$ and a real $\gamma > 0$ there exists a real $\beta=\beta(k,\gamma)$ so that the following holds:
	If a $k$-uniform hypergraph $H$ on $m$ vertices has the following properties for some $t$:\\
	(i) $(1-\beta)t < d(x) < (1+\beta)t$ holds for all vertices;\\
	(ii) $d(x,y) < \beta t$ for all distinct $x$ and $y$;\\
	then the hypergraph has a matching covering at least $m(1-\gamma)$ vertices. \qed
\end{lemma}

Fix some fixed tournament $Q$ on $q \ge 3$ vertices (we shall later specify $Q$). For an $n$-vertex tournament $T$ define the hypergraph $H(T,Q)$  as follows.
Its vertex set are all the $\binom{n}{2}$ edges of $T$ and its edges are all the (edges of the) $Q$-copies in $T$. Notice that $H(T,Q)$ is $\binom{q}{2}$-uniform. Now suppose that $T \sim {\mathcal T}_n$.
The next lemma shows that whp $H(T,Q)$ has a large subhypergraph that satisfies the conditions of Lemma \ref{l:fr}.
\begin{lemma}\label{l:hyp-properties}
	Let $p_Q$ be the probability that a fixed random tournament on $q$ vertices is isomorphic to $Q$.
	Let $T \sim {\mathcal T}_n$. Whp $H(T,Q)$ has an induced subhypergraph on at least $\binom{n}{2}-3n^{1.9}$ vertices with the following properties:\\
	(i) The degree of each vertex of $H'$ is $(1+o(1))p_Q\binom{n-2}{q-2}$\,;\\
	(ii) The co-degree of each pair of vertices is at most $n^{q-3}$.
\end{lemma}
\begin{proof}
	We start with the second assertion, which is not probabilistic. Consider two vertices of $H(T,Q)$, i.e., two distinct edges of $T$, say $(u,v)$ and $(w,z)$. The total number of $q$-subsets of vertices of $T$ that contain both of these edges is $\binom{n-4}{q-4}$ if $\{u,v\} \cap \{w,z\} = \emptyset$ and is
	$\binom{n-3}{q-3}$ if $\{u,v\} \cap \{w,z\} \neq \emptyset$. In any case, we see that the number of $q$-sets
	of vertices containing both $(u,v)$ and $(w,z)$ is less than  $n^{q-3}$ and in particular, the co-degree of $(u,v)$ and $(w,z)$ in $H(T,Q)$, which only counts $q$-sets that induce a copy of $Q$, is less than $n^{q-3}$.
	
	For the first assertion, fix a vertex of $H(T,Q)$, i.e., an edge $e=(u,v)$ of $T$.
	Let $d(e)$ be the random variable corresponding to the degree of $e$ in $H(T,Q)$.
	Let ${\mathcal X}$ be the $q$-sets of vertices of $T$ that contain both $u$ and $v$
	and observe that $|{\mathcal X}|=\binom{n-2}{q-2}$. For $X \in {\mathcal X}$, consider the
	indicator random variable $d(X)$ for the event that $T[X]$ is isomorphic to $Q$.
	We have that $d(e) = \sum_{X \in {\mathcal X}} d(X)$ and clearly
	$d(X)  \sim {\rm  Bernoulli}(p_Q)$, so we obtain that
	$$
	{\mathbb E}[d(e)] = p_Q\binom{n-2}{q-2}\;.
	$$
	We show that $d(e)$ is concentrated by considering its variance. To this end, fix two elements
	of ${\mathcal X}$, say $X$ and $Y$, and consider ${\rm Cov}(d(X),d(Y))$. Notice that as each of $X$ and $Y$ contain both $u$ and $v$, we have that $|X \cap Y| \ge 2$. Now, if $|X \cap Y| \ge 3$ we shall use the trivial bound ${\rm Cov}(d(X),d(Y)) \le 1$ (recall that $d(X)$ and $d(Y)$ are indicators). Suppose now that $|X \cap Y| = 2$, i.e., they to not have common elements other
	than $u$ and $v$. Now, suppose that we are given $d(X)$, i.e., we are given the information whether $T[X]$ is isomorphic to $Q$.
	Moreover, suppose we are even revealed the orientation of the edge connecting the two elements of $X \cap Y$.
	This provides no information as to whether $T[Y]$ is isomorphic to $Q$ (knowing the orientation of a
	single edge of a tournament says nothing about the structure of that tournament).
	In particular, ${\rm Cov}(d(X),d(Y))=0$.
	We now have
	\begin{align*}
		{\rm Var}[d(e)] & \le {\mathbb E}[d(e)]+ 2\sum_{t=2}^{q-1} \sum_{\substack{X,Y \in {\mathcal X}\\ |X \cap Y|=t}} {\rm Cov}(d(X),d(Y))\\
		& = {\mathbb E}[d(e)]+ 2\sum_{t=3}^{q-1} \sum_{\substack{X,Y \in {\mathcal X}\\ |X \cap Y|=t}} {\rm Cov}(d(X),d(Y))\\
		& \le {\mathbb E}[d(e)]+ 2\sum_{t=3}^{q-1} n^{2q-2-t}\\
		& \le {\mathbb E}[d(e)]+ 2qn^{2q-5}\\
		& \le 3qn^{2q-5}
	\end{align*}
	where in the last inequality we have used that $q \ge 3$. We may now apply Chebyshev's inequality and obtain that
	$$
	\Pr\left[\,|d(e) - {\mathbb E}[d(e)] \,| \ge n^{q-2.1} \right] \le \frac{3qn^{2q-5}}{n^{2q-4.2}}= \Theta(n^{-0.8})\;.
	$$
	As $T$ has fewer than $n^2$ edges, we obtain from the last inequality and from Markov's inequality that whp, all but $O(n^{1.2}) < n^{1.5}$ vertices $e$ of $H(T,Q)$ have
	$|d(e) - {\mathbb E}[d(e)]| \le n^{q-2.1}$.
	Consider then the set $F$ of at most $n^{1.5}$ vertices $e$ of $H(T,Q)$ which violate the last inequality.
	For a vertex $v \in V(T)$ we say that $v$ is {\em dangerous} if $v$ is an endpoint of at least $n^{0.6}$
	elements of $F$. So, there are fewer than $2n^{0.9}$ dangerous vertices. Remove from $H(T,Q)$ all elements
	of $F$ and also remove all vertices $e$ of $H(T,Q)$ such that $e$ has an endpoint which is a dangerous vertex of $T$.
	Let $H'$ be the induced subhypergraph of $H(T,Q)$ obtained after the removal.
	As we remove at most $|F|+2n^{1.9}$ vertices from $H(T,Q)$, we have that $H'$ contains at least
	$\binom{n}{2}-3n^{1.9}$ vertices. 
	 Clearly, the co-degree of any two vertices in $H'$ is
	not larger than it is in $H(T,Q)$. By how much might a degree of a vertex $e=(u,v)$ in $H'$ decrease?
	It might belong to a copy of $Q$ which contains a dangerous vertex of $T$, but there are only at most
	$2n^{0.9}n^{q-3} = 2n^{q-2.1}$ such copies. As $u$ and $v$ are non-dangerous, there may additionally be at most $2n^{0.6}$ vertices $(x,y)$ of $H(T,Q)$ where $(x,y) \in F$ and $|\{x,y\} \cap \{u,v\}|=1$.
	But then these may cause a further reduction of at most $2n^{0.6}n^{q-3} < n^{q-2.1}$ in  the degree of $e$.
	Additionally, it may be that $(u,v)$ belongs to copies of $Q$ which contain an element $(x,y)$ of $F$ such that $\{x,y\} \cap \{u,v\}=\emptyset$. But then these may cause a further reduction of at most $n^{1.5}n^{q-4} < n^{q-2.1}$ in the degree of $e$. It now follows that all the vertices of $H'$ 
	have degree 
	${\mathbb E}[d(e)] \pm 5n^{q-2.1}$, which is ${\mathbb E}[d(e)](1+o(1))$.
\end{proof}
\begin{corollary}\label{coro:matching}
	Whp, $T \sim {\mathcal T}_n$ contains $(1-o(1))\binom{n}{2}/\binom{q}{2}$ pairwise edge disjoint copies of $Q$.
\end{corollary}
\begin{proof}
	We apply Lemma \ref{l:fr} to the subhypergraph $H'$ of Lemma \ref{l:hyp-properties}. It has
	$m \ge \binom{n}{2}-3n^{1.9}$ vertices, it is $\binom{q}{2}$-regular, the degrees of all vertices are
	$(1+o(1))p_Q\binom{n-2}{q-2}$, while the co-degrees are negligible, as they are only at most $n^{q-3}$.
	Hence, by Lemma \ref{l:fr}, $H'$ has a matching which covers all but $o(n^2)$ vertices of $H'$
	hence all but $o(n^2)$ vertices of $H(T,Q)$. But observe that by the definition of $H(T,Q)$ such a matching corresponds to a set of pairwise edge-disjoint copies of $Q$ in $T$.
\end{proof}

For a tournament $Q$ on $q$ vertices, for a $2$-edge coloring $c$ of it, and for a vertex $v$, let $s(Q,c,v)$
denote the number of vertices (other than $v$) that are monochromatically dominated by $v$ using paths of length at most $2$. Let $s(Q,c)$ be the sum of $s(Q,c,v)$ taken over all vertices, and let $s(Q)$ be the minimum of $s(Q,c)$ taken over all $2$-edge colorings.
As an example, $s(C_3)=4$ as can be seen by the non-monochromatic coloring of $C_3$, while $s(Q)=\binom{q}{2}$ when $Q$ is a transitive tournament.
\begin{lemma}\label{l:sq}
	Let $q \ge 3$ be fixed and let $Q$ be a tournament on $q$ vertices. Then whp,
	$f^*(2,2,n)+o(n) \ge \frac{s(Q)}{q(q-1)}\cdot n$
\end{lemma}
\begin{proof}
	Suppose $T \sim {\mathcal T}_n$. By Corollary \ref{coro:matching}, whp $T$ contains
	$(1-o(1))\binom{n}{2}/\binom{q}{2}$ pairwise edge disjoint copies of $Q$. Consider any $2$-edge coloring
	$c$ of $T$. Restricted to any copy of $Q$ in $T$, we have that $s(Q,c) \ge s(Q)$. Summing over all copies, this contributes at least $(1-o(1))s(Q)\binom{n}{2}/\binom{q}{2}$ to the number of vertices monochromatically dominated by each vertex using paths of length at most $2$, summed over all vertices. Hence, there is a vertex that monochromatically dominates at least $(1-o(1))s(Q)\frac{n-1}{q(q-1)}$ other vertices
	using paths of length at most $2$, whence $f^*(2,2,n)+o(n) \ge \frac{s(Q)}{q(q-1)}\cdot n$.
\end{proof}

Lemma \ref{l:sq} triggers a search for a tournament $Q$ for which $\frac{s(Q)}{q(q-1)}$ is as large as possible.
In particular, we would like this ratio to be larger than $\frac{2}{3}$ in order to show that
$f^*(2,2,n)$ is truly superior (whp) to $f(2,2,n)$, even if the latter assumes CH (recall Theorem \ref{t:1}).
Apparently no tournament of order $8$ or less achieves more than this $2/3$ ratio; for example, see Figure \ref{f:tour-4} showing that $s(Q) \le 8$ for each of the four non-isomorphic tournaments on four vertices.
Fortunately, however, five distinct tournaments on $9$ vertices (of a total of $191536$ tournaments)
have $s(Q)=49$. All of these tournaments are regular (there are $15$ regular tournaments on $9$ vertices, but $10$ of them have $s(Q) \le 48$). The adjacency matrix of one of these five tournaments is depicted in
Table \ref{table:1}. A code computing $s(Q)$ for tournaments up to order $9$ can be obtained from
\url{https://github.com/raphaelyuster/path-monochromatic/blob/master/path-monochromatic.cpp}. 

\begin{figure}[t]
	\includegraphics[scale=0.6,trim=34 375 150 38, clip]{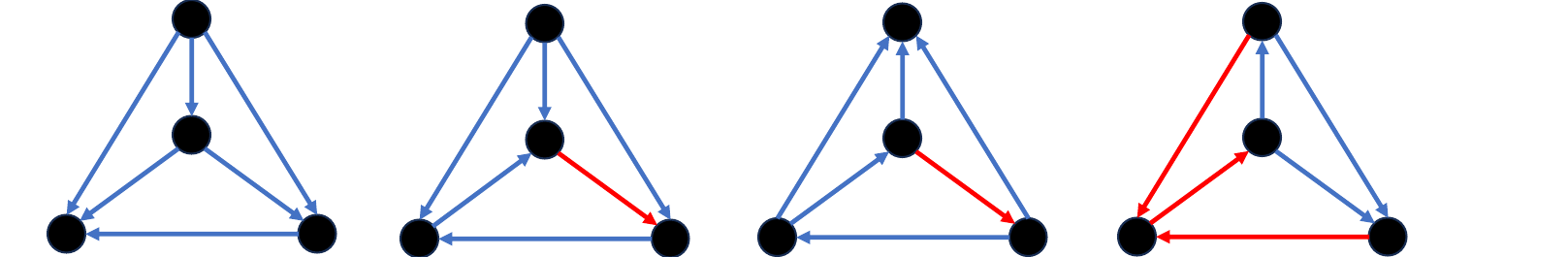}
	\caption{All tournaments $Q$ on four vertices and a red/blue coloring of each, showing that $s(Q) \le 8$ for each.}
	\label{f:tour-4}
\end{figure} 

\begin{table}
	\centering
	\begin{tabular}{ccccccccc}
		$0$ & $1$ & $1$ & $1$ & $1$ & $0$ & $0$ & $0$ & $0$ \\
		$0$ & $0$ & $1$ & $0$ & $1$ & $0$ & $1$ & $0$ & $1$ \\
		$0$ & $0$ & $0$ & $1$ & $1$ & $0$ & $1$ & $1$ & $0$ \\
		$0$ & $1$ & $0$ & $0$ & $0$ & $1$ & $0$ & $1$ & $1$ \\
		$0$ & $0$ & $0$ & $1$ & $0$ & $1$ & $0$ & $1$ & $1$ \\
		$1$ & $1$ & $1$ & $0$ & $0$ & $0$ & $0$ & $0$ & $1$ \\
		$1$ & $0$ & $0$ & $1$ & $1$ & $1$ & $0$ & $0$ & $0$ \\
		$1$ & $1$ & $0$ & $0$ & $0$ & $1$ & $1$ & $0$ & $0$ \\
		$1$ & $0$ & $1$ & $0$ & $0$ & $0$ & $1$ & $1$ & $0$
	\end{tabular} 
	\caption{The adjacency matrix of a tournament $Q$ on $9$ vertices with $s(Q)=49$.
	A value of $1$ in row $i$ column $j$ represents a directed edge from $i$ to $j$.}\label{table:1}
\end{table}  

The following corollary now follows from Lemma \ref{l:sq} and from the tournament whose adjacency matrix is given in Table \ref{table:1}.
\begin{corollary}
	Whp, $f^*(2,2,n)+o(n) \ge \frac{49n}{72}$. \qed
\end{corollary}

We have now completed the proofs of both parts of Theorem \ref{t:2} and both parts of Theorem \ref{t:3}.

\subsection*{Acknowledgement}

The author thanks both reviewers for insightful suggestions.

\end{document}